\documentclass{amsart}
\usepackage{amsmath,amssymb,amsthm}
\usepackage{enumitem}

\newcommand{\pair}[1]{\langle #1\rangle}
\newcommand{\powerw}{\mathcal{P}(\omega)}
\newcommand{\cont}{\mathfrak{c}}
\newcommand{\filt}{\mathcal{F}}
\newcommand{\A}{\mathcal{A}}
\newcommand{\B}{\mathcal{B}}
\newcommand{\T}{\mathcal{T}}
\newcommand{\cl}[2][X]{\mathrm{cl}_{#1}\!\left(#2\right)}
\renewcommand{\b}{\mathfrak{b}}

\newcommand{\poset}{\mathcal{P}}

\newtheoremstyle{theorem}
     {11pt}
     {11pt}
     {}
     {}
     {\bfseries}
     {}
     {.5em}
     {\noindent\thmnumber{#2}. \thmname{#1}\thmnote{#3}}

\theoremstyle{theorem}

\newtheorem{thm}{Theorem}[section]
\newtheorem{lemma}[thm]{Lemma}
\newtheorem{defi}[thm]{Definition}
\newtheorem{coro}[thm]{Corollary}
\newtheorem{propo}[thm]{Proposition}
\newtheorem{ques}[thm]{Question}
\newtheorem{claim}[thm]{Claim}

\title{Pseudoradial spaces and copies of $\omega_1+1$}
\author[Bella]{Angelo Bella}
\author[Dow]{Alan Dow}
\author[Hern\'andez]{Rodrigo Hern\'andez-Guti\'errez}

\address[Bella]{Department of Mathematics and Computer Science, University of Catania, Citt\'a universitaria, viale A. Doria 6, 95125 Catania, Italy}
\email[Bella]{bella@dmi.unict.it}
\address[Dow]{Department of Mathematics, UNC-Charlotte, 9201 University City Blvd., Charlotte, NC 28223-0001}
\email[Dow]{adow@uncc.edu}
\address[Hern\'andez]{Departamento de Matem\'aticas, Universidad Aut\'onoma Metropolitana campus Iztapalapa, Av. San Rafael Atlixco 186, Col. Vicentina, Iztapalapa, 09340, Mexico city, Mexico}
\email[Hern\'andez]{rodrigo.hdz@gmail.com}


\begin{document}
 
 \begin{abstract}
  In this paper we compare the concepts of pseudoradial spaces and the recently defined strongly pseudoradial spaces in the realm of compact spaces. We show that $\mathrm{MA}+\cont=\omega_2$ implies that there is a compact pseudoradial space that is not strongly pseudoradial. We essentially construct a compact, sequentially compact space $X$ and a continuous function $f:X\to\omega_1+1$ in such a way that there is no copy of $\omega_1+1$ in $X$ that maps cofinally under $f$. We also give some conditions that imply the existence of copies of $\omega_1$ in spaces. In particular, $\mathrm{PFA}$ implies that compact almost radial spaces of radial character $\omega_1$ contain many copies of $\omega_1$.
 \end{abstract}
 
 \maketitle

 \section{Introduction}
 
 All spaces are assumed to be Hausdorff.
 
 Recall that a topological space $X$ is \emph{pseudoradial} if for every non-closed subset $A\subset X$ there is a point $x\in\overline{A}\setminus A$ and a transfinite sequence $\langle x_\alpha\rangle_{\alpha<\kappa}$ with range in $A$ and converging to $x$.
 
 The systematic investigation on the topological properties of pseudoradial spaces was initiated by Arhangel'ski\u\i{} more than 40 years ago. Since then, several subclasses of pseudoradial spaces have been considered by many authors. 
 
 Recently the further notion of strongly pseudoradial spaces appeared in the literature \cite{brazas-fabel}. All ordinals in this paper will have the order topology when considered as topological spaces.
 
 \begin{defi}
  A topological space $X$ is called strongly pseudoradial if for any non-closed subset $A\subset X$ there is a limit ordinal $\gamma$ and a continuous map $f:\gamma+1\to X$ such that $f[\gamma]\subset A$ and $f(\gamma)\notin A$.
 \end{defi}
 
 In \cite{brazas-fabel} the authors pointed out that, without any loss of generality, in the above definition the ordinal $\gamma$ can be assumed to be a regular cardinal and the function $f$ injective.
 
 Roughly speaking, the difference between pseudoradial and strongly pseudoradial spaces consists in replacing transfinite converging sequences with compact ordinals.
 
 As $\omega+1$ is a compact ordinal, we immediately see that every sequential space is strongly pseudoradial, but this is the only case when we can easily determine if a space is of that kind. The passage from $\omega+1$ to the successor of an uncountable cardinal appears much more difficult. However, a remarkable consequence of the proper forcing axiom ($\mathrm{PFA}$) describes a possibility to do it for $\kappa=\omega_1$ (see also \cite[Theorem 5.14]{nyikos}).
 
 \begin{thm}\cite{ctble-tight-PFA}
  $\mathrm{PFA}$ implies that in every countably compact regular space of character at most $\omega_1$, the closure of a subset $A$ can be obtained by first adding all limits of convergent sequences and then adding to the resulting set $\hat{A}$ all points $x$ for which there is a copy $W$ of $\omega_1$ in $\hat{A}$ such that $W\cup\{x\}$ is homeomorphic to $\omega_1+1$.
 \end{thm}

 \begin{coro}\label{PFA-positive}
  $\mathrm{PFA}$ implies that every countably compact regular space of character at most $\omega_1$ is strongly pseudoradial.
 \end{coro}
 
 On the other hand, the one-point compactification of Ostaszewski's space shows that a compact pseudoradial space may fail to be strongly pseudoradial. This follows from hereditary separability of Ostaszewki's space.
 
 A natural question is then whether it is possible to obtain the conclusion of Corollary \ref{PFA-positive} by weakening the topological hypothesis. First, we were able to produce a counterexample under Martin's axiom.

\begin{thm}\label{counterex}
$\b=\cont=\omega_2$ implies there is a compact pseudoradial space that is not strongly pseudoradial. 
\end{thm}

The \emph{radial character} of a pseudoradial space $X$ is the smallest cardinal $\kappa$ such that the definition of pseudoradiality for $X$ works by taking only transfinite sequences of length not exceeding $\kappa$. Thus, the following question is natural after considering Corollary \ref{PFA-positive}.

\begin{ques}\label{main-question}
 Assume $\mathrm{PFA}$. Is every compact Hausdorff pseudoradial space of radial character at most $\omega_1$ strongly pseudoradial?
\end{ques}

We were unable to answer Question \ref{main-question}. However, we can achive a partial positive result for a special class of pseudoradial spaces, as we now explain.

A sequence $\{x_\alpha:\alpha<\kappa\}$ converging to a point $x$ is called \emph{thin} if for any $\beta<\kappa$ we have that $x\notin\overline{\{x_\alpha:\alpha<\beta\}}$. A space is called \emph{almost radial} if in the usual definition of pseudoradiality we replace ``sequence'' with ``thin sequence''.

Our counterexample $X$ from Theorem \ref{counterex} is a compact, sequentially compact space $X$ with a point $\rho\in X$ of character $\omega_1$ such that there are no countable sequences or copies of $\omega_1$ converging to $\rho$ (see Theorem \ref{thm-bc}).

In Section \ref{section-copies-omega1} we include several results on existence of copies of $\omega_1$. In particular, we highlight the following result that contrasts with our counterexample from Theorem \ref{counterex}. 

\begin{thm}\label{main-positive}
  Assume $\mathrm{PFA}$. Let $X$ be a compact almost radial space of radial character at most $\omega_1$. Then every point of $X$ is either the limit of a countable sequence or the limit of a copy of $\omega_1$.
 \end{thm}

The search of copies of $\omega_1$ goes back to \cite[Problem 1.3]{nyikos-handbook} where Nyikos asks whether there exists a first countable, countably compact, non-compact space which does not contain a copy of $\omega_1$. Nyikos himself proved that a consequence of $\diamondsuit$, which is in fact compatible with $\mathrm{MA}+\cont>\omega_1$, gives a positive answer to his question (see \cite[19.1]{fremlin}). In \cite{fremlin} Fremlin, assuming $\mathrm{PFA}$,  gives sufficient conditions for the existence of copies of $\omega_1$ for spaces that can be mapped onto $\omega_1$ and gives a series of applications. This question was also considered under PFA in \cite{ctble-tight-PFA}, where other conditions for the existence of copies of $\omega_1$ are given. Later, Eisworth and Nyikos in \cite{eis-nyi} give a model of $\mathrm{CH}$ in which every first countable, countably compact, non-compact space contains a copy of $\omega_1$. Also, in \cite{eisworth-perf_preim} Eisworth shows that any perfect pre-image of $\omega_1$ with countable tightness contains a closed copy of $\omega_1$.

\section{$\T$-algebras}\label{talgebradef}
 
 In order to construct the space from Theorem \ref{counterex}, we will use Koszmider's notion of $\T$-algebra from \cite{kosz}. $\T$-algebras are special kinds of the minimally generated  Boolean  algebras  first  studied  by  Koppelberg \cite{min-gen}. All our Boolean algebras will be subalgebras of $\powerw /\mathsf{fin}$ with the order relation $\subset\sp\ast$ of almost inclusion. Section 3 of \cite{dow-pichardo} contains a a thorough analysis of the following discussion.
 
 Given a Boolean algebra $A\subset\powerw$ and $x\subset\omega$, the Boolean algebra generated by $A\cup\{x\}$ is
 $$
 A(x)=\{(a_0\cap x)\cup(a_1\setminus x):a_0,a_1\in A\}.
 $$
 
 Let $A\subset\powerw$ be a Boolean algebra and $u$ an ultrafilter of $A$. An element $x\subset\omega$ is called \emph{minimal} for $\pair{A,u}$ if $u$ is the only ultrafilter in $A$ that does not generate an ultrafilter in the Boolean algebra $A(x)$. Notice that in this case, $\omega\setminus x$ is also minimal for $\pair{A,u}$. 
 
 Let $\lambda\leq\cont$ and $\A=\{a_{\alpha+1}:\alpha<\lambda\}\subset[\omega]\sp\omega$. For each $\alpha\leq\lambda$, define $B_\alpha$ to be the Boolean algebra generated by $\{a_{\beta+1}:\beta<\alpha\}$. We will say that $\A$ is a \emph{coherent minimal sequence} if for every $\alpha<\lambda$ the filter $u_{\alpha}$ in $B_\alpha$ generated by $\{a_{\beta+1}:\beta<\alpha\}$ is an ultrafilter and $a_{\alpha+1}$ is minimal for $\pair{B_\alpha,u_{\alpha}}$. 

  Let us describe the Stone space of $B_\lambda$. Consider $\alpha<\lambda$. Let $x_\alpha$ the filter in $B_\lambda$ generated by $u_\alpha\cup\{\omega\setminus a_{\alpha+1}\}$. Since $a_{\alpha+1}$ is minimal for $\pair{B_\alpha,u_\alpha}$, $x_\alpha\cap B_{\alpha+1}$ is an ultrafilter (thus, proper) in $B_{\alpha+1}$. By recursion it is possible to show that, $x_\alpha\cap B_{\beta}$ is in fact an ultrafilter in $B_\beta$ for all $\alpha<\beta\leq\lambda$. It is not hard to conclude that the Stone space of $B_\lambda$ is equal to $X_\lambda=\{u_\lambda\}\cup\{x_\alpha:\alpha<\lambda\}$. 

  Let $\alpha<\lambda$. Then the clopen set defined by $\omega\setminus a_{\alpha+1}$ in the Stone space misses $\{x_\beta:\alpha<\beta<\lambda\}\cup\{u_\lambda\}$. Thus, the segment $\{x_\beta:\beta\leq\alpha\}$ is open. This easily implies that the Stone space of $B_\lambda$ is scattered.
  
  Next, we define $\T$-algebras. First, given $t\in2\sp{<\cont}$ such that $\mathsf{dom}(t)=\alpha+1$ for some $\alpha$, we define $t\sp\star=(t\!\!\restriction_\alpha)\sp\frown(1-t(\alpha))$. Each $\T$-algebra will be defined by using a tree. A subtree $T$ of $2\sp{<\cont}$ is called \emph{acceptable} if the following two conditions hold:
  \begin{enumerate}[label=(\roman*)]
   \item the domain of each member of $T$ is a successor ordinal,
   \item if $t\in T$ and $\alpha<\mathsf{dom}(t)$, then $t\!\!\restriction_{\alpha+1}\in T$, and
   \item for all $t\in2\sp{<\cont}$, $t\in T$ if and only if $t\sp\star\in T$.
  \end{enumerate}

  For each $p\in2\sp{\leq\cont}$, $o(p)$ will denote its order type. Given an acceptable tree $T\subset2\sp{<\cont}$, a $T$-algebra is a Boolean algebra generated by a sequence $\{a_t:t\in T\}\subset[\omega]\sp\omega$ such that the following properties hold
  \begin{enumerate}[label=(\alph*)]
   \item given $t\in T$, $\{a_{t\restriction{\alpha+1}}:\alpha+1<o(t)\}$ is a coherent minimal sequence, and
   \item given $t\in T$, $a_t=\omega\setminus a_{t\sp\star}$.
  \end{enumerate}
  
  It turns out that the Stone space of a $\T$-algebra is easy to describe. In fact, the set of ultrafilters of a $T$-algebra $\A=\{a_t:t\in T\}$ is in one-to-one correspondence with the set $bT$ of branches of $T$. Given $p\in bT$, $\{a_{p\restriction_{\alpha+1}}:\alpha<\mathsf{dom}(p)\}$ generates the ultrafilter that corresponds to $p$. So without loss of generality, we will identify the Stone space of $\A$ with $bT$.
  
  It also turns out that some of the topological properties of $bT$ can be checked by looking at the topology generated by the branches. For every branch $p\in bT$, by property (a) in the definition of $T$-algebra, there is a topological space $X_p=\{x_{p\restriction\alpha}:\alpha<o(p)\}\cup\{p\}$ defined by the coherent minimal sequence $\{a_{p\restriction{\alpha+1}}:\alpha<o(p)\}$ as described above.
  
  Given $p,q\in bT$ with $p\neq q$, we define $p\wedge q$ to be the largest common predecesor of $p$ and $q$. Thus, $o(p\wedge q)$ is the ordinal $\alpha$ such that $p\!\!\restriction_\alpha=q\!\!\restriction_\alpha$ but $p(\alpha)\neq q(\alpha)$.
  
  Fix $p\in bT$. Then there is a continuous function $\pi_p:bT\to X_p$ that projects the Stone space of the $T$-algebra onto the branch space. For every $q\in bT\setminus\{p\}$, $\pi_p(q)=x_{p\restriction{o(p\wedge q)}}$ and $\pi_p(p)=p$. Equivalently, for every $\alpha<o(p)$
  $$
  \pi_p\sp\leftarrow(x_{p\restriction{\alpha}})=\{q\in bT:o(p\wedge q)=\alpha\}.
  $$
  The relation we are interested in is summarized with the following results.
  
  \begin{lemma}\cite[Proposition 3.4]{dow-cclosed}
   Let $p\in bT$, $S\subset bT\setminus\{p\}$. Then $p\in\cl[bT]{S}$ if and only if $p\in\cl[X_p]{\{\pi_p(q):q\in S\}}$.
  \end{lemma}
  
  \begin{lemma}\cite[Proposition 3.4]{dow-cclosed}\label{branch-conv}
   Let $p\in bT$ and $\{p_n:n<\omega\}\subset bT\setminus\{p\}$. Then $\{p_n:n<\omega\}$ converges to $p$ in $bT$ if and only if $\{\pi_p(p_n):n<\omega\}$ converges to $p$ in $X_p$.
  \end{lemma}

  Since we are interested in copies of $\omega_1$, we would like to find a necessary condition for the existence of copies of $\omega_1$ that can be checked in branches and avoided through careful construction of the $T$-algebra. 
  
  \begin{lemma}\label{branch-omega1}
   Let $f:\omega_1+1\to bT$ be an embedding. Then there is a closed, unbounded set $S\subset\omega_1$ such that $(\pi_{f(\omega_1)}\circ f)\!\!\restriction{\!S}$ is one-to-one and increasing. Thus, $X_{f(\omega_1)}$ contains a copy of $\omega_1+1$.
  \end{lemma}
  \begin{proof}
   Let $p=f(\omega_1)$. First, notice that the fibers of points of $X_p\setminus\{p\}$ under $\pi_p\circ f$ are all countable. Otherwise, there is an uncountable $T\subset\omega_1$ and $\beta<o(p)$ such that $(\pi_p\circ f)[T]=\{x_{p\restriction{\beta}}\}$. This would imply that $a_{p\restriction{\beta+1}}$ and $\omega\setminus a_{p\restriction{\beta+1}}$ define clopen sets that separate $p=(\pi_p\circ f)(\omega_1)$ from $(\pi_p\circ f)[T]$ in $X_p$, which contradicts the continuity of $\pi_p\circ f$.
   
  Thus, there is an increasing enumeration $\{\xi(\alpha):\alpha<\omega_1\}$ of the set $\{\beta<o(p):\exists \alpha<\omega_1 ( \pi_p(f(\alpha))=x_{p\restriction\beta})\}$. For each $\alpha<\omega_1$, let $T_\alpha=\{\beta<\omega_1:\pi_p(f(\beta))=x_{p\restriction\xi(\alpha)}\}$.  We shall recursively define an increasing injective function $\sigma:\omega_1\to\omega_1$ and an element $s_\alpha\in T_{\sigma(\alpha)}$ for each $\alpha<\omega_1$.
  
  For every non-limit ordinal $\alpha<\omega_1$, let $\sigma(\alpha)=\min(\omega_1\setminus\sigma[\alpha])$ and let $s_\alpha$ be the first element of  $T_{\sigma(\alpha)}$. Now assume that $\alpha<\omega_1$ is a limit ordinal. Define $s_\alpha=\sup\{s_\beta:\beta<\alpha\}<\omega_1$, then there exists $\gamma<\omega_1$ such that $s_\alpha\in T_\gamma$. We then define $\sigma(\alpha)=\gamma$. 
  
  Assume that for some $\delta<\alpha$ we have that $\gamma\leq\sigma(\delta)$.  By continuity, $\{(\pi_p\circ f)(s_\beta):\beta<\alpha\}$ converges to $(\pi_p\circ f)(s_\alpha)$. However, initial segments are open in $X_p$ so $\{x_{p\restriction{\beta}}:\beta\leq\xi(\sigma(\delta))\}$ is an open set that contains $(\pi_p\circ f)(s_\alpha)$ and misses $\{(\pi_p\circ f)(s_\beta):\delta<\beta<\alpha\}$. This is a contradiction so we obtain that $\sigma(\alpha)>\sigma(\delta)$ for all $\delta<\alpha$.
  
  This completes the construction and it easily follows that the set $S=\{s_\alpha:\alpha<\omega_1\}$ has the properties required.
 \end{proof}
  
  In our arguments below, we will construct our acceptable tree $T$ by recursion. So we will have the typical situation where we have constructed a tree $T\sp\prime$ which will eventually be a subtree of $T$. Clearly, at this point of the construction, it is possible to consider the space of branches $bT\sp\prime$ of $T\sp\prime$. We will use the notation we defined for the branch space on $bT\sp\prime$. Also, notice that if $q\in bT\sp\prime$ and $p\in bT$ is such that $q\subset p$, then there is a projection map $\pi_q\sp{p}:X_p\to X_q$ where $\pi_q\sp{p}(x)=x$ if $x\subset q$ and $\pi_q\sp{p}(x)=q$ if $q\subset x\subset p$.

  \section{Proof of Theorem \ref{counterex}}
  
  We will assume $\b=\cont=\omega_2$ in this section. In order to construct the space required in Theorem \ref{counterex}, we will construct a sequence $\{t_\alpha:\alpha<\cont\}\subset 2\sp{<\cont}$ with the property that for all $\beta<\cont$ there is $\alpha<\beta$ and $i\in 2$ such that $t_\beta=(t_\alpha)\sp\frown{i}$. Then our acceptable tree will be $$T=\{(t_\alpha)\sp\frown{i}:\alpha<\cont,i\in 2\}.$$ We will also recursively define a $T$-algebra $\{a_t:t\in T\}$ (with the notation of Section \ref{talgebradef}) along with the definition of the tree. Then the branch space $bT$ will be the space we are looking for. Clearly, $bT$ is compact.
  
  For $\alpha<\omega_1$, define $t_\alpha\in 2\sp\alpha$ to be such that $t_\alpha(\beta)=0$ for all $\beta<\alpha$ and let $\{a_{t_{\alpha+1}}:\alpha<\omega_1\}$ be any strictly $\subset\sp\ast$-decreasing sequence of infinite subsets of $\omega$. We remark that by the discussion above, both $(t_\alpha)\sp\frown0=t_{\alpha+1}$ and $(t_\alpha)\sp\frown1=t_{\alpha+1}\sp\ast$ will be in $T$, and both $a_{t_{\alpha+1}}$ and $a_{t_{\alpha+1}\sp\ast}=\omega\setminus a_{t_{\alpha+1}}$ are defined, for every $\alpha<\omega_1$. Let $\rho\in2\sp{\omega_1}$ be such that $\rho(\beta)=0$ for all $\beta<\omega_1$.
  
  Clearly, the topology defined on $X_{\rho}=\{x_{t_\alpha}:\alpha<\omega_1\}\cup\{\rho\}$ is the order topology. Once we have the tree $T$ completely defined, if $p\in bT$ is such that $\rho\subset p$, then the function $\pi\sp{p}_\rho\circ\pi_p$ maps $bT$ onto $X_\rho$.We make a record of this fact as follows.
  
  \begin{claim}
  $X_{\rho}$ is naturally homeomorphic to $\omega_1+1$.
   \end{claim}
   
   For sake of notational simplicity, for each $\alpha<\omega_1$, the filter $x_{t_\alpha}$ will be denoted by $y_\alpha$. We also notice that $bT$ has uncountable tightness.
 
 \begin{claim}\label{tightness}
 There exists $p\in bT$ such that $\rho\subset p$ and $p$ has uncountable tightness in $bT$.
 \end{claim}
 \begin{proof}
   For each $\alpha<\omega_1$, choose any $q_\alpha\in bT$ such that $t_{\alpha+1}\subset q_\alpha$ and $q_\alpha(\alpha+1)=1\neq t_{\alpha+2}(\alpha+1)$. Define $A=\{q_\alpha:\alpha<\omega_1\}$ and let $p\in bT$ be any complete accumulation point of $A$. We claim that there is no countable subset of $A$ with $p$ in its closure.
   
   First, notice that $p$ extends $\rho$. Otherwise, there is some $\beta<\omega_1$ such that $t_{\alpha+1}\subset p$ for $\alpha<\beta$ and $p(\beta)\neq t_{\beta+1}(\beta)$. Then $\omega\setminus a_{s_{\beta+1}}$ is in the ultrafilter defined by $p$ but not in the ultrafilter defined by $q_\alpha$ for all $\beta\leq\alpha<\omega_1$. Thus, we obtain a contradiction.
   
   Now, let $N\subset\omega_1$ be countable, we will show that $\{q_\alpha:\alpha\in N\}$ does not have $p$ in its closure. Let $\beta<\omega_1$ be an upper bound of $N$. If $\alpha\in N$, by the definition of $q_\alpha$, $\omega\setminus a_{t_{\alpha+1}}$ is in the ultrafilter defined by $q_\alpha$. Since $a_{t_{\beta+1}}\subset\sp\ast a_{t_{\alpha+1}}$, $\omega\setminus a_{s_{\beta+1}}$ is in the ultrafilter defined by $q_\alpha$. Thus, $a_{t_{\beta+1}}$ gives a neighborhood of $q$ that does not contain $q_\alpha$, for all $\alpha\in N$. Thus, the conclusion follows.  
  \end{proof}  

 So we are left to define $\{t_\alpha:\omega_1\leq\alpha<\cont\}$ and we will do this by recursion. Given $\omega_1\leq\alpha<\cont$, let $T_\alpha$ be the nodes of the tree that have been defined before step $\alpha$, that is, the set $\{(t_\beta)\sp\frown{i}:\beta<\alpha,i\in 2\}$. On step $\alpha\in\cont\setminus\omega_1$, we will choose $t_\alpha$ among $bT_\alpha$, the branches of $T_\alpha$ (notice that all branches are of length $<\cont$), and define $a_{(t_\alpha)\sp\frown{0}}$ and $a_{(t_\alpha)\sp\frown{1}}$. Consider a surjective function $e=(e_0,e_1):\cont\setminus\omega_1\to(\cont\setminus\omega_1)\times\cont$ such that for all $\alpha<\cont$, $e_0(\alpha)<\alpha$.
  
  It is known that under $\cont\leq\omega_2$, every sequentially compact, compact space is pseudoradial. This was proved by \v Sapirovski\u{\i} under $\mathrm{CH}$ \cite{shapirovski} and by Juh\'asz and Szentmikl\'ossy under $\cont=\omega_2$ \cite{juhasz-sz-pseudoradial}. Thus, in order to make the space pseudoradial, it is enough to make it sequentially compact.
  
  Let $\omega_1\leq\alpha<\cont$. Let $\{f_{\pair{\alpha,\beta}}:\beta<\cont\}\subset {}\sp{\omega}T_\alpha$ be an enumeration of all increasing $\omega$-sequences of elements of $T_\alpha$. In other words, $f_{\pair{\alpha,\beta}}:\omega\to T_\alpha$ is such that $f_{\pair{\alpha,\beta}}(n)\subset f_{\pair{\alpha,\beta}}(n+1)$ for all $n<\omega$.  We will require the following inductive assumption.

 \begin{quote}
 $(a)_\alpha$ For every $\omega_1\leq\beta<\alpha$, there are $A_\beta\in[\omega]\sp\omega$ and $q\in bT_\alpha$ such that $f_{e(\beta)}[A_\beta]\subset q$ and $\{x_{f_{e(\beta)}(n)}: n\in A_\beta\}$ converges to $q$ in $X_q$.
 \end{quote}
  
  \begin{claim}\label{seqcomp}
   $(a)_\alpha$ for all $\omega\leq\alpha<\omega_1$ implies that $bT$ is sequentially compact.
  \end{claim}
  \begin{proof}
   By Lemma \ref{branch-conv}, $bT$ is sequentially compact if and only if for every increasing $\omega$-sequence $f$ of elements of $T$, there is a branch $p$ extending $f$ such that some subsequence of $f$ converges to $p$ in $X_p$. Indeed, $f=f_{e(\beta)}$ for some $\beta<\omega_1$, and this implies that $\{x_{f_{e(\beta)}(n)}: n\in A_\beta\}$ converges to some $p\in bT$. To see this, notice that $(a)_\alpha$ is preserved under limits in the following sense: if $\lambda\leq\cont$ is a limit ordinal and $\{q_\alpha:\alpha<\lambda\}\subset T$ are such that $\{x_{f_{e(\beta)}(n)}: n\in A_\beta\}$ converges to $q_\alpha$ in $X_{q_\alpha}$ for all $\alpha<\lambda$, then $\{x_{f_{e(\beta)}(n)}: n\in A_\beta\}$ converges to $q=\bigcup\{q_\alpha:\alpha<\lambda\}$ in $X_q$.
  \end{proof}
  
 Next we would like to add some inductive hypothesis so that at the end of the construction, we obtain a point with no copies of $\omega_1$ converging to it. Our strategy will be to never split the branch $\rho\in bT_{\omega_1}$ so that it remains a branch of $bT$ at the end of the construction, and kill all possible copies of $\omega_1$ converging to it. According to Lemma \ref{branch-omega1}, we can test convergence to $\rho$ just by looking at copies of $\omega_1$ contained in $X_\rho$. Also, all copies of $\omega_1$ contain copies of the ordinal 
 $$\omega\sp{\cdot2}+1=\{\omega\cdot n+m:n,n<\omega\}$$
which is homeomorphic to the one-point compactification of the free union of countably many convergent sequences. 

So the strategy will be to consider all copies of $\omega\sp{\cdot2}+1$ contained in $X_\rho$ and make sure that they are not lifted to copies of $\omega\sp{\cdot2}+1$ in $bT$. Thus, let $\{s(\alpha):\alpha<\cont\}$ be the set of all continuous bijective functions with domain $\omega\sp{\cdot 2}+1$ and image contained in $\omega_1$. We require the following inductive assumption.

\begin{quote}
 $(b)_\alpha$ For every $\omega_1\leq\beta<\alpha$, there are $C_\beta\in[\omega]\sp\omega$, a function $\varphi_\beta\in{}\sp\omega\omega$, and $t\in T_\alpha$ with $t_{s(\omega\sp{\cdot 2})}\subset t$ such that $a_t\in y_{s(\omega\cdot n)}$  and $a_{t\sp\star}\in y_{s(\omega\cdot n+\varphi_\beta(n))}$ for every $n\in C_\beta$.
 \end{quote}  
 
 So assume that we are in step $\alpha<\cont$ of the construction. We need to choose $t_\alpha$, $A_\alpha$, $C_\alpha$, $\varphi_\alpha$ and define the partition $\omega=a_{(t_{\alpha})\sp\frown0}\cup a_{(t_{\alpha})\sp\frown1}$.
 
 First, we explain how to choose $A_\alpha$. Let $q_\emptyset=\bigcup f_{e(\alpha)}$ and $B_\emptyset=\omega$; then it is easy to see that in $X_{q_\emptyset}$, $\{x_{f_{e(\alpha)}(n)}: n\in B_\emptyset\}$ converges to $q_\emptyset$. If there is $q\in bT_{\alpha}$ such that $\{x_{f_{e(\alpha)}(n)}: n\in B_\emptyset\}$ converges to $q$ in $X_{q}$, we define $A_\alpha=\omega$. Otherwise, there exists $q_\emptyset\sp\prime\in T_\alpha$ such that $\{x_{f_{e(\alpha)}(n)}: n\in B_\emptyset\}$ converges to $q_\emptyset\sp\prime$ in $X_{q_\emptyset\sp\prime}$ but for $i\in 2$, $\{x_{f_{e(\alpha)}(n)}: n\in B_\emptyset\}$ does not converge to $(q_\emptyset\sp\prime)\sp\frown i$ in $X_{(q\sp\prime_\emptyset)\sp\frown i}$. Let $q_{i}=(q_\emptyset\sp\prime)\sp\frown i$ for $i\in 2$. Then by the definition of $\T$-algebra, there is a partition $B_\emptyset=\omega=B_0\cup B_1$ such that $\{x_{f_{e(\alpha)}(n)}: n\in B_i\}$ converges to $q_i$ for $i\in 2$. Continuing in this fashion, by recursion on ${}\sp{<\omega}\omega$, we try to construct a sequence of nodes $\{q_s:s\in{}\sp{<\omega}\omega\}\subset T_\alpha$ and a sequence of sets $\{B_s:s\in{}\sp{<\omega}\omega\}\subset[\omega]\sp\omega$ such that $\{x_{f_{e(\alpha)}(n)}: n\in B_s\}$ converges to $q_s$ in $X_{q_s}$. If there is some $s$ such that there is a branch $q\in bT_\alpha$ with $\{x_{f_{e(\alpha)}(n)}: n\in B_s\}$ converging to $q$ in $X_{q}$, we define $A_\alpha=B_s$. Otherwise, given $s\in{}\sp{<\omega}\omega$, we can always choose incompatible $q_{s\sp\frown 0}$, $q_{s\sp\frown 1}$ above $q_s$ and a partition $B_s=B_{s\sp\frown 0}\cup B_{s\sp\frown 1}$ as required. Assume that we never stopped in the construction (otherwise, we are done). For each $\varphi\in{}\sp\omega\omega$, let $q_\varphi=\cup\{q_s:s\subset\varphi\}$ and let $B_\varphi$ be any pseudointersection of $\{B_s:s\subset \varphi\}$; it easily follows that $\{x_{f_{e(\alpha)}(n)}: n\in B_\varphi\}$ converges to $q_\varphi$ in $X_{q_\varphi}$. The set $\{B_\varphi:\varphi\in{}\sp\omega\omega\}$ is of size $\cont$ and $\lvert T_\alpha\rvert<\cont$, so this means that there is $\psi\in{}\sp\omega\omega$ such that $q_\psi\in bT_\alpha$. Define $A_\alpha=B_\psi$ and let $r=q_\psi$. Then $\{x_{f_{e(\alpha)}(n)}: n\in A_\alpha\}$ converges to $r\in bT_\alpha$ in $X_r$.
 
 The next step is to choose $t_\alpha\in bT_\alpha$ and $C_\alpha\in[\omega]\sp\omega$. Consider the copy of $\omega\sp{\cdot 2}$ given by $\{y_{s(\alpha)(\xi)}:\xi\in\omega\sp{\cdot 2}\}$ in $X_\rho$. By an argument similar to the one in the previous paragraph, it is possible to find $C_\alpha\subset\omega$ and some $t_\alpha\in bT_\alpha$ such that $\bigcup s(\alpha)\subset t_\alpha$ and the sequence $\{y_{s(\alpha)(\omega\cdot n)}:n\in C_\alpha\}$ converges to $t_\alpha$ in $X_{\alpha}$.
 
 Recall that according to $(a)_{\alpha+1}$, we need to preserve the convergence of each sequence $f_{e(\beta)}\!\!\restriction_{A_\beta}$ for all $\beta\leq\alpha$. Since we are choosing to split $t_\alpha$, we only need to worry about those sequences with $f_{e(\beta)}[A_\beta]\subset t_\alpha$ (and in fact only those that converge to $t_\alpha$). 
 
 Thus, we have to define $\varphi_\alpha:\omega\to\omega\setminus\{0\}$ in such a way that the sequence $\{y_{s(\alpha)(\omega\cdot n+\varphi_\alpha(n))}:n\in B_\alpha\}$ is almost disjoint with each of $f_{e(\beta)}[A_\beta]$ for all $\beta\leq\alpha$. Consider the set
 $$
 S=\left\{\beta\leq\alpha:\sup (f_{e(\beta)}[A_\beta])= s(\alpha)(\omega\sp{\cdot 2})\right\}
 $$
 
 Let $\beta\leq\alpha$. If $\beta\notin S$, there is nothing to worry about. Otherwise, there exists a function $\psi_\beta:\omega\to\omega$ such that for every $n<\omega$, 
 $$
 \left\{y_{s(\beta)(\omega\cdot n+m)}: \psi_\beta(n)\leq m\right\}\cap\left \{f_{e(\beta)}(n):n\in A_\beta\right\}=\emptyset.
 $$
 
 Consider also the following result the proof of which, being standard, we shall omit.
 
 \begin{lemma}
 Assume that we have a topological space with underlying set $\omega\sp{\cdot 2}+1$ that satisfies the following properties.
 \begin{enumerate}
 \item Every point of the form $\omega\cdot n+m$ with $n<\omega$ and $0<m<\omega$ is isolated.
 \item For each $n<\omega$, $\{\omega\cdot n+m:m<\omega\}$ converges to $\omega\cdot (n+1)$.
 \item $\{\omega\cdot n:n<\omega\}$ converges to $\omega\sp{\cdot 2}$.
 \item $\omega\sp{\cdot 2}$ has character strictly less than $\b$.
 \end{enumerate}
 Then there exists $\varphi:\omega\to\omega\setminus\{0\}$ such that $\{\omega\cdot n+\varphi(n):n<\omega\}$ converges to $\omega\sp{\cdot 2}$.
 \end{lemma}
 
This implies that we can choose $\varphi_\alpha:C_\alpha\to\omega\setminus\{0\}$ such that $\{y_{s(\alpha)(\omega\cdot n+\varphi_\alpha(n))}:n\in C_\alpha\}$ converges to $t_\alpha$ in $X_{t_\alpha}$. 

By $\b=\omega_2$ we can further assume that for every $\beta\in S$ we have that $\{n\in C_\alpha:\varphi_\alpha(n)\leq\psi_\beta(n)\}$ is finite. Thus, we obtain that  $\{y_{s(\alpha)(\omega\cdot n+\varphi_\alpha(n))}:n\in C_\alpha\}$ is almost disjoint from $f_{e(\beta)}[A_\beta]$ for all $\beta\leq\alpha$.

So all that remains is to define the partition $\omega=a_{(t_{\alpha})\sp\frown0}\cup a_{(t_{\alpha})\sp\frown1}$ in such a way that $(a)_{\alpha+1}$ and $(b)_{\alpha+1}$ hold. For the sake of notational simplicity, let $z_n=y_{s(\alpha)(\omega\cdot n+\varphi_\alpha(n))}$ for all $n\in C_\alpha$. Given $a\in B_{t_\alpha}\setminus u_{t_\alpha}$, we will denote its associated clopen set as
$$
a\sp\ast=\{x_t:t\subset t_\alpha,a\in x_t\}.
$$

For each $n\in C_\alpha$, since $z_n$ is a point of first countability of $X_{t_\alpha}$ so let $\{c(n,m):m<\omega\}\subset B_{t_\alpha}$ define a local open base at $z_n$. 

Assume $\beta<o(t_\alpha)$. The point $x_{t_\alpha\restriction\beta}$ is not a limit point of $\{z_n:n\in C_\alpha\}$ (in $X_{t_\alpha}$). By normality, there are open sets $U_\beta$ and $V_\beta$ with $U_\beta\cap V_\beta=\emptyset$ such that $x_{t_\alpha\restriction\beta}\in U_\beta$ and $\{z_n:n\in C_\alpha\}\subset\sp\ast V_\beta$. Then there exists a function $g_\beta:C_\alpha\to\omega$ such that $\{n\in C_\alpha:c(n,g_\beta(n))\sp\ast\not\subset V_\beta\}$ is finite.

Now, let $\omega_1\leq\beta\leq\alpha$. We know that $\{z_n:n\in C_\alpha\}$ is almost disjoint with $f_{e(\beta)}[A_\beta]$ so we can find a function $h_\beta:C_\alpha\to \omega$ such that $\{n\in C_\alpha:f_{e(\beta)}[A_\beta]\cap c(n,g_\beta(n))\sp\ast\neq\emptyset\}$ is finite. Also, there exists a function $h:C_\alpha\to\omega$ such that for all $n<\omega$, $y_{s(\alpha)(\omega\cdot n)}\notin c(n,h(n))$.

The set of functions
$$
\{g_\beta:\beta<o(t_\alpha)\}\cup\{h_\beta:\omega_1\leq\beta\leq\alpha\}\cup\{h\}
$$
is of size $<\b$ so there exists $g:C_\alpha\to\omega$ that bounds them all mod finite. We obtain an open set $W=\bigcup\{c(n,g(n))\sp\ast:n\in C_\alpha\}$ of $X_{t_\alpha}$ with the following properties:
\begin{enumerate}[label=(\roman*)]
 \item $\{z_n:n\in C_\alpha\}\subset W$,
 \item the only limit point of $W$ in $X_{t_\alpha}$ is $t_\alpha$,
 \item $\{y_{s(\alpha)(\omega\cdot n)}:n\in C_\alpha\}\cap W=\emptyset$, and
 \item for every $\omega_1\leq\beta<\omega_1$ such that $f_{e(\beta)}[A_\beta]\subset t_\alpha$ then one of the two following conditions holds:
 \begin{enumerate}
 \item $\{x_{f_{e(\beta)}(n)}:n<\omega\}\cap W$ is finite, or
 \item $\{x_{f_{e(\beta)}(n)}:n<\omega\}$ does not converge to $t_\alpha$ in $X_{t_\alpha}$ .
 \end{enumerate}
\end{enumerate}

Thus, we define
$$
a_{(t_\alpha)\sp\frown0}=\bigcup\{c(n,g(n)):n\in C_\alpha\},
$$
and $a_{(t_\alpha)\sp\frown1}=\omega\setminus a_{(t_\alpha)\sp\frown0}$. From properties (i) to (iv) above it is easy to see that both $(a)_{\alpha+1}$ and $(b)_{\alpha+1}$ will hold. Thus, we have finished our construction.

  Notice that in our construction, the branches that we split in every step have subsequences of $X_\rho$ converging to them. Thus, we can infer the following.
  
  \begin{claim}\label{rhopreserved}
  The branch $\rho\in bT_{\omega_1}$ is never split, so $\rho\in bT_\alpha$.
  \end{claim}
  
By inductive hypothesis $(a)_\alpha$ for all $\alpha\in\cont\setminus\omega_1$ we obtain that $bT$ is sequentially compact (Claim \ref{seqcomp}). Since $\cont=\omega_2$ we obtain, as discussed above, that $bT$ is pseudoradial.

Finally, we prove that $bT$ is not strongly pseudoradial. We will prove that this property fails at $\rho\in bT$. If $bT$ were strongly pseudoradial, there would be an infinite cardinal $\kappa$ and an embedding $f:\kappa+1\to T$ such that $f[\kappa]\subset bT\setminus\{\rho\}$ and $f(\kappa)=\rho$. By Claims \ref{tightness} and \ref{rhopreserved}, and by the fact that $\rho$ has character $\omega_1$ in $bT$, we obtain that $\kappa=\omega_1$. By Lemma \ref{branch-omega1}, we may assume that $\pi_\rho\circ f$ is injective. Since $\pi_\rho\circ f$ is continuous, it is an embedding so $(\pi_\rho\circ f)[\omega_1]$ is a copy of $\omega_1$ contained in $X_\rho$. So in particular, there is a copy of $\omega\sp{\cdot 2}+1$ contained in $X_\rho$, we may assume that the first one in our enumeration is $\{y_{s(\gamma)(\xi)}:\xi\in\omega\sp{\cdot 2}+1\}$. According to Lemma \ref{branch-conv}, there is $q\in bT$ such that $y_{s(\gamma)(\omega\sp{\cdot 2})}\subset q$ and $\{y_{s(\gamma)(\xi)}:\xi\in\omega\sp{\cdot 2}\}\cup\{q\}\subset X_q$ is homeomorphic to $\omega\sp{\cdot 2}+1$ . But then, according to $(b)_{\gamma+1}$, it easily follows that the sets $\{y_{s(\omega\cdot n)}:n\in C_\gamma\}$ and $\{y_{s(\omega\cdot n+\varphi_\gamma(n))}:n\in C_\gamma\}$ are separated by clopen sets of $X_q$. This is a contradiction so indeed $bT$ is not strongly pseudoradial. 

\section{Forcing copies of $\omega_1$}\label{section-copies-omega1}

Here we give a generalization of several results (\cite{fremlin}, \cite{ctble-tight-PFA} and \cite{eisworth-perf_preim}) concerning the existence of a proper forcing that forces a copy of $\omega_1$.

Let $X$ be any completely regular countably compact non-compact
space with a base  $\B$ of open sets such that $X\in\B$ and $\emptyset\notin\B$. For a subset
$H$ of $X$, we consider the $\omega$-closure of $H$
$$\cl[\omega]{H} =
\bigcup \{ \overline{a} : a\in [H]^{\leq \aleph_0}\}.$$
Say that $H$ is $\omega$-closed if $\cl[\omega]{H}= H$. 
 
Suppose that $\filt$ is a countably
complete maximal free filter of $\omega$-closed subsets of $X$.
Choose any regular cardinal $\kappa$ such that $X,\B,\filt,\omega_1\in H(\kappa)$. 

\begin{defi} 
For any countable set $M$, we define the trace of the filter $\filt$ as
$$\mbox{Tr}(\filt, M)=\bigcap \{ \overline{F\cap M} : F\in M\cap \filt\}.$$
\end{defi}

We will need the following fact. We refer the reader to \cite{dow-elem_submodels} for the use of elementary submodels in topology.

\begin{lemma}\label{Fplus}
  For any countable elementary submodel  $M\prec H(\kappa)$,
  such that $\filt\in M$,
  then for any subset  $H$ of $X$ that is in $M$, then
  \begin{enumerate}
    \item if $H\cap \mbox{Tr}(\filt, M)$ is not empty,        
        then   $H\in \filt^+$,
    \item if  $H\in \filt^+$ then  $\overline{H\cap M}$ contains $\mbox{Tr}(\filt,M)$.
  \end{enumerate}
\end{lemma}
\begin{proof}
  Assume that $H\in M$ and that  $H\cap \mbox{Tr}(\filt,M)\neq
  \emptyset$. Since $\mbox{Tr}(\filt,M) \subset (\bigcap \mathcal  F)\cap M$, it follows that $H\cap F\cap M\neq\emptyset$ for all $F\in \filt\cap M$. By elementarity, it follows that $H\cap F\neq\emptyset$ for all $F\in\filt$ so $H\in\filt\sp+$. By the maximality of $\filt$ and elementarity, $\cl[\omega]{H}\in \filt\cap M$.  Again by elementarity, $\cl[\omega]{H}\cap M$ is contained
   in $\overline{H\cap M}$. This completes the proof.
\end{proof}

We next define a poset $\poset_{X,\B,\filt,\kappa}$.

\begin{defi} A  condition $p\in\poset_{X,\B,\filt,\kappa}$ is a function with domain
  $\mathcal M_p$ such that, for each $M, M_1\in \mathcal M_p$,
  \begin{enumerate}
   \item  $p(M) = \langle x_{p(M)}, \mathcal U_{p(M)},\filt_{p(M)}\rangle$  is an element of $X\times [\B]^{<\aleph_0}\times [\filt]^{<\aleph_0}$,
   \item  $x_{p(M)}\in\textrm{Tr}(\mathcal F, M)$,
\item $\mathcal M_p$ is a finite $\in$-chain of countable elementary submodels of $H(\kappa)$, each containin $\omega_1$, $X$, $\B$, and $\filt$,
  \item if $M\in M_1$, then $p(M)\in M_1$.
  \end{enumerate}
  We define $p\leq q$ provided:
  \begin{enumerate}
    \setcounter{enumi}{4}
  \item $\mathcal M_q\subset \mathcal M_p$,
  \item for each $M\in \mathcal M_q$, $x_{p(M)} = x_{q(M)}$,
     $\filt_{q(M)} = \filt_{p(M)}$, and
    $\mathcal U_{q(M)}\subset \mathcal U_{p(M)}$,
  \item for each $M_1 \in \mathcal M_q$ and $M\in (\mathcal M_p\setminus\mathcal{M}_q)\cap M_1$ such that $\mathcal M_q\cap M_1\in M$,
    then $x_{p(M)}\in U$ for any $U$
    such that $x_{q(M_1)}\in U$ and $U\in \mathcal U_{q(M_2)}$ 
for some $M_2\in \mathcal M_q$ with $M_1\subset M_2$.      
  \end{enumerate}
\end{defi}

We can rephrase the complicated last condition with the help
of the following definitions:
 for $q\in \poset$ and $M\in \mathcal
 M_q$, let $$\mathcal U(q,M) = \mathcal U_{x(q(M))} \cap
 \bigcup \{ \mathcal U_{q(M_1)} : M\subset M_1\in \mathcal M_q\}.$$
Then let $W(q,M) = \bigcap \mathcal U(q,M)$ if $\mathcal U(q,M)\neq\emptyset$ and $W(q,M)=X$ otherwise.  
An alternative way to state the last condition is that
\begin{itemize}
 \item[(7')] if $M\in \mathcal{M}_p\setminus \mathcal M_q$ and if $M_1$ is the $\in$-minimal element of $\mathcal M_q$ containing $M$, then $x_{p(M)}\in W(q,M_1)$.
\end{itemize}
  We can notice that it then follows that
 $W(p,M)\subset W(q,M_1)$. This is a key property to have to ensure
that $\poset_{X,\B,\filt,\kappa}$ is transitive.

We do not include a proof that $\poset_{X,\B,\filt,\kappa}$ is proper because we will give a proof of a stronger statement, Theorem \ref{suitable-proper} below. However, we do prove the following.

\begin{propo}\label{forcing-omega1}
If $G\subset \poset_{X,\B,\filt,\kappa}$ is a generic filter, then\label{cub}:
  \begin{enumerate}
  \item  $\mathcal M_G = \bigcup \{ \mathcal M_p : p\in G\}$ is an
    uncountable $\in$-chain, 
\item   $C_G = \{ M\cap \omega_1 : M\in \mathcal M_G\}$
  is a cub subset of $\omega_1$, and
\item  for each $M\in \mathcal M_G$ such that $M\cap \omega_1 = \delta$
is a limit point of $C_G$,
$\filt\cap M = \bigcup \{ \filt\cap M' : M'\in
\mathcal M_G\cap M\}$.
  \end{enumerate}  
\end{propo}
\begin{proof}
Let us use $\poset$ to denote $ \poset_{X,\B,\filt,\kappa}$. We omit the easy proof that $\mathcal M_G$  and $C_G$ 
are uncountable.   For each countable $M\prec H(\kappa)$ such
that $\filt\in M$, let $\{ F(M,n) : n\in \omega\}$ be an
enumeration of $\filt\cap M$.
\medskip

\noindent Claim 1:   For each $\delta\in\omega_1$ and $n\in\omega$,
 the following set is dense in $\poset$:
 $$
 \begin{array}{lcl}
  D_{\delta,n} & = & \{ p\in \mathcal  P \colon   (\exists M\in \mathcal M_p) \ \textrm{such that either}\\
   & & \left(\delta\subset M\textrm{ and}\  (\forall q\leq p)~\left(\bigcup(M\cap \mathcal M_q)  = \bigcup(M\cap \mathcal M_p)\right)\right)\ \textrm{or}\\
  & & ( M\cap \omega_1 = \delta\textrm{ and}\ (\exists M_1\in\mathcal{M}_p\cap M)\ (F(M,n)\in\filt_{p(M_1)}) )\}.
 \end{array}
 $$

\noindent Informally speaking, the first condition asserts that $\delta$ is not a limit of $C_G$; the second condition asserts that $\delta\in C_G$ and that the element $F(M,n)$ of $M$ will appear in every $M\sp\prime\in\mathcal{M}_G$ with $M_1\in M\sp\prime\in M$.
 \medskip
 
\noindent Proof of Claim 1: Let $p_1\in \poset$ be arbitrary.
It is easy to extend $p_1$ so as to ensure that $\delta\in \bigcup \mathcal M_{p_1}$.  Let $\delta^*$ be the minimum of the set
  $\{ M\cap \omega_1 : (\exists p^* < p_1)~~ M\in \mathcal M_{p^*}  \ \textrm{and}\ \delta\leq M\cap \omega_1\}$. Choose $p_2 \leq p_1$
  such that $\delta^* \in \{ M\cap \omega_1 : M\in \mathcal
  M_{p_2}\}$.

  Suppose first that $\delta^* = \delta$ and let $M$ be such that $\delta^*=M\cap\omega_1$.  If $p_2\notin D_{\delta,n}$   then the first clause in the definition must fail. Therefore we may choose some $q\leq p_2$ such that $\mathcal M_q\cap M$ is not empty  and let $ M_1$ be the maximum element of  $\mathcal M_q\cap M$.    Simply define $q^* $ where the only change from $q$ is 
  that $F(M,n)$ is in $\mathcal F_{q^*(M_1)}$. 
  Since $ M_1$ is the maximum element of $\mathcal M_q\cap M$
  and since $F(M,n)$ is evidently in $M$, it follows that 
$q^*\in  \poset$ and through routine checking that $q^*$ is also
  below $p_2$.   Therefore $q^*$ is in $D_{\delta,n}$.

Now we assume that $\delta < \delta^*$ and again that
 $p_2$ is not in $D_{\delta,n}$. Let
 $M$ be the element of $ \mathcal M_{p_2}$
with $M\cap \omega_1 = \delta^* > \delta$ and note that
(by the failure of the first clause)
there is a $q\leq p_2$ such that $\bigcup (M\cap \mathcal M_q)
\neq \bigcup(M\cap \mathcal M_{p_2})$. Choose such a $q$
and again let 
$M_1 $ be the maximum element of $M\cap \mathcal M_{q}$.
Fix any strictly descending sequence $\{ F_\alpha : \alpha\in
\omega_1\}\subset \mathcal F$ that is an element of $M_1$.
Note that $F_\delta\in M$ since
$\{ F_\alpha : \alpha\in
\omega_1\}\in M$.
We now define an  extension $p$ of $p_2$ that is in
$D_{\delta,n}$.  Set $\mathcal M_p $ equal to
$\mathcal M_{p_2}\cup \{ \strut M_1\}$ and
$p_2\subset p$.  To define $p$ we just have to choose
the value for $p(M_1)$.
We let $p(M_1) = \langle x_{q(M_1)},
\mathcal U_{q( M_1)} , \{ F_\delta\}\rangle$.
Since we already have that $q$ is an extension of $p_2$, it
is routine to check that $p$
is also   an extension of $p_2$.  We check that
$p$ satisfies the first condition of $D_{\delta,n}$.
If
$q<p$ and $M_1\in M'\in \mathcal M_q\cap M$,
then $p(M_1)\in M'$, implying that $\delta\in M'$,
and this contradicts the definition of $\delta^*$.
\medskip

\noindent Claim 2:  If $G\subset \poset$ is a filter that meets
  $D_{\delta,0}$ for all $\delta\in \omega_1$, then
   $C = \{ M\cap \omega_1 : (\exists p\in G) ~ M\in \mathcal M_p\}$ is
  a closed and unbounded subset of $\omega_1$.

  \medskip
  
Proof of Claim 2: For $\delta\in\omega_1$,
the  fact that $G$ meets $D_{\delta,0}$ implies that
 $C\setminus\delta$
is not empty. To show that $C$ is closed
we assume $\delta\notin C$ and show that it 
is not a limit point of $C$
by showing 
that $C\cap \delta$ has a maximum element. 
Choose $p\in G\cap D_{\delta,0}$ and let $M\in \mathcal M_p$
be as in the definition of $D_{\delta,0}$.
Since $\delta$ is  not in $C$, $ \delta \in M\ $.
If $\mathcal M_p\cap M$ is empty, let $\beta = 0$,
otherwise let $\bar M$ be the maximum element of $\mathcal M_p\cap M$,
and let $\beta = \bar M\cap \omega_1$. It now follows that
for all $q\leq p$ in $G$ and $M'\in \mathcal M_q\cap M$,
then $M'\cap \omega_1$ is less than or equal to $\beta$.
It thus  follows that $C$ is disjoint from
 the interval $(\beta,\delta)$.\medskip
 
\noindent Claim 3: If $G$ is a filter that meets $D_{\delta,n}$ for all $\delta\in\omega_1$ and $n\in\omega$, then condition (3) will hold.\medskip

Let $M\in\mathcal M_G$ such that $\delta=M\cap\omega_1$ is a limit. We will prove that each element $F(M,n)\in\filt\cap M$ is in some $M\sp\prime\in\mathcal{M}_G\cap M$. We may choose $p\in G\cap D_{\delta,n}$ such that $M\in\mathcal{M}_p$.

First, we argue that the first clause in $D_{\delta,n}$ does not hold. Since $\delta$ is a limit of $C_G$, $M\cap\mathcal{M}_G$ is infinite. But $M\cap\mathcal{M}_p$ is finite so there must exist $q\in G$ such that $M\cap\mathcal{M}_q\not\subset M\cap\mathcal{M}_p$. Any common extension of $p$ and $q$ contradicts the first clause.

Thus, the second clause holds. Let $M_1\in\mathcal{M}_p\cap M$ such that $F(M,n)\in\filt_{p(M_1)}$. Since $\delta$ is a limit of $C_G$, there exists $M\sp\prime\in\mathcal{M}_G$ be such that $M_1\cap\omega_1<M\sp\prime\cap\omega_1<\delta$. Let $q\in\poset$ with $q\leq p$ and $M\sp\prime\in\mathcal{M}_q$. Then $M_1\in M\sp\prime$ are elements of $\mathcal{M}_q$, by the last condition in the definition of $q\in\poset$ it follows that $q(M_1)\in M\sp\prime$. This implies that $\filt_{p(M_1)}=\filt_{q(M_1)}\in M\sp\prime$ so $F(M,n)\in M\sp\prime$. 
\end{proof}

Now, assume that $G\subset\poset$ is a filter. This implicitly defines a function $$f:C_G\to\{ x_{p(M)}:p\in G,\ M\in \mathcal M_p\}$$ where $M\cap \omega_1\in C_G$  is sent to $x_{p(M)}$. Moreover, if $G$ is generic, then $C_G$ is homoemorphic to $\omega_1$ by (b) in Proposition \ref{forcing-omega1}.

\begin{lemma}\label{homeomorphism}
If $G\subset\poset_{X,\B,\filt,\kappa}$ is a generic filter, then
  $X$ contains a copy of $\omega_1$ with the subspace topology inherited from $\B$.
\end{lemma}
\begin{proof}
 Given $x\in X$ and $U\in\mathcal{U}_x$, the set
 $$
 E_{x,U}=\{p\in\poset_{X,\B,\filt,\kappa}:\exists M\in\mathcal{M}_p\ (x=x_{p(M)})\Rightarrow\exists M\sp\prime\in\mathcal{M}_p\setminus M\ (U\in\mathcal{U}_{p(M\sp\prime)})\}
 $$
 is easily seen to be dense in $\poset_{X,\B,\filt,\kappa}$.
 
 Now we prove that $f$ is continuous, let $\delta\in C_G$ and $x_\delta = f(\delta)$. It is enough to prove that every time $U\in\B$ with $x_\delta\in U$ there exists $\beta<\delta$ such that $\{x_\alpha:\beta<\alpha\leq\delta\}\subset U$.
 
 This is clearly true when $\delta$ is not a limit, so assume in the following that $\delta$ is a limit. Let $p\in G\cap E_{x,U}$ be such that there is $M\in\mathcal{M}_p$ with $x_\delta=x_{p(M)}$. We may assume that $\mathcal{M}\cap M\neq\emptyset$ and let $M_1=\bigcup(\mathcal{M}\cap M)$. Define $\beta=M_1\cap\omega_1$. If $\beta<\alpha<\delta$, there is $q\in G$ and $M_2\in\mathcal{M}_q$ with $M_2\cap\omega_1=\alpha$. We may assume that $q\leq p$. Since $p\in E_{x,U}$, $U\in\mathcal{U}_{p(M\sp\prime)}$ for some $M\sp\prime\in\mathcal{M}_p\setminus M$. By the last condition in the definition of $q\leq p$, it follows that $x_{q(M_2)}\in U$.
 
 This shows that $f:C_G\to X$ is continuous. Recall that $C_G$ is homeomorphic to $\omega_1$. Even if $f$ is not a homeomorphism, we claim that some restriction of $f$ is an embedding of $\omega_1$ to $X$. Indeed, consider the \v Cech-Stone extension $\beta f:\beta C_G\to\beta X$. The only case in which $\beta f$ is not an embedding is if it is not injective and this can only happen if there is $\alpha\in C_G$ such that $\beta f(\omega_1)=f(\alpha)$. Such an $\alpha$ is unique because $f$ is injective. So $\beta f\restriction{(C_G\setminus\alpha)}=f\restriction{(C_G\setminus\alpha)}$ is an embedding. Since $C_G\setminus\alpha$ is homeomorphic to $\omega_1$, the statement of the lemma follows.
\end{proof}

Now we formulate a strong generalization that encompasses $\mathrm{PFA}$ results such those for first countable spaces (\cite{fremlin}), or spaces of countable tightness (\cite{eisworth-perf_preim}), or even spaces with character at most $\omega_1$(\cite{ctble-tight-PFA}).

\begin{defi} 
Let $X$ be a countably compact space. A function $\varphi$ will be called suitable if the domain of $\varphi$ is the set of all closed subsets of $X$, and for all $B\in\mathrm{dom}(\varphi)$, $\overline{\varphi(B)} = B$. 
\end{defi}

\begin{defi}  
 If $X$ is countably compact, $\varphi$ is a suitable function on $X$, and $\filt$ is a maximal free filter of closed subsets of $X$, then the poset $\poset_{X,\B,\filt,\kappa}^\varphi$ is the
  subposet of $\poset_{X,\B,\filt,\kappa}$ consisting of all   those $p\in\poset_{X,\B,\filt,\kappa}$ satisfying that $p(M)\in \varphi(\mbox{Tr}(\mathcal F, M))$ for all $M\in \mathcal{M}_p$.
\end{defi}

For example, if $X$ is $\omega$-bounded and of countable tightness, then (under PFA) $\varphi$ may be all those points of relative countable character (\cite{eisworth-perf_preim}).

\begin{thm}\label{suitable-proper}
If $\varphi$ is suitable, then $\poset_{X,\B,\filt,\kappa}\sp\varphi$ is proper. 
\end{thm}
\begin{proof}
 Let
  $\poset  = \poset_{X,\B,\filt,\kappa}$ and let $\poset\in H(\theta)$ for a regular cardinal $\theta$. Also let $M$ be a countable elementary submodel of $H(\theta)$ such that $\poset\in M$.  Then to prove that $\poset$ is proper,
  it suffices to prove that for any $p\in \poset$ with $M\cap H(\kappa)\in \mathcal M_p$, then $p$ is an $(M,\poset)$-generic condition.  Equivalently,  if  $D \in M$ is any dense open subset of $\poset$, we must show there is an $r\in D\cap M$ that is compatible with $p$.
  
  By extending $p$  we may assume that $p\in D$. Let $\mathcal M_p \setminus M$ be  enumerated in increasing order as $\{ M^p_0,\cdots, M^p_{\ell-1}\}$. Note that $M^p_0 = M\cap H(\kappa)$. An important property of $\poset$ is that $\bar p = p\restriction (\mathcal M_p\cap M)$ is itself an element of $\poset$ and is in $M$.
    
  Say that $q\in \poset$ end-extends $\bar p$ if $q\leq \bar p$ and $\mathcal M_{\bar p}$ is an initial segment of $\mathcal M_q$. Now let $D_\ell$ be the set of $q\in D$ that  end-extend $\bar p$ and satisfy that  $\lvert\mathcal M_q\setminus \mathcal M_{\bar p}\rvert =\ell$. It follows that $D_\ell\in M$. For $q\in D_\ell$, let $\mathcal M_q \setminus \mathcal M_{\bar p}$ be enumerated as $\{ M^q_i  : i<\ell\}$ and let $\{ x^q_i : i<\ell\}$ enumerate  $\{ x_{q(M^q_i)} : i<\ell\}$. We leave  the reader to verify  that if $r\in D_\ell\cap M$ is such that        $\{ x^r_i : i<\ell\}$ is a subset of $W(p,M^p_0)$, then $r$ is compatible with $p$. So, it becomes our task to show there is such an $r\in D_\ell\cap M$.
  
  In order for this proof to work, we have to argue in $M^p_0$ rather than in $M$. To this end, we work with the set $$T_\ell = \{ \vec x^q  = \langle x^q_i : i<\ell\rangle : q \in D_\ell\}.$$ It is immediate that $T_\ell\in M\cap H(\kappa)= M^p_0$. For any $\vec x^q \in T_\ell$ and $0<j<\ell$, let $\vec x^q \restriction j $ denote         $\langle x^q_i : i<j\rangle$ and $\vec x^q\restriction 0 $          is the empty sequence.
  
  Of course $T_\ell\subset X^\ell$, we will recursively define a  sequence $T_j \subset X^j$ for $j<\ell$. For any $j<\ell$ and tuple $\vec x\in T_j$, we let $$H(\vec x, T_{j+1})= \{ y \in X : \vec x^\frown y \in T_{j+1}\}.$$ Then, by recursion,  $$T_j = \{ \vec x \in X^j : H(\vec x, T_{j+1}) \in \mathcal F^+\}.$$ This recursion is definable in $M^p_0$, hence for any $j<\ell$ and $\vec x\in M^p_j$, $H(\vec x, T_{j+1})$ is an element of  $M^p_j$. It recursively follows from Lemma \ref{Fplus} that
  $\vec x^p\restriction j\in T_j$ for each $j<\ell$. This means that the empty sequence is an element
  of $T_0\cap M^p_0$, implying that $H(\emptyset, T_1) \in \mathcal F^+\cap M^p_0$. 
  
  By Lemma \ref{Fplus}, $\mbox{Tr}(\mathcal F,M^p_0)$ is contained in the closure of $H(\emptyset, T_1)\cap M^p_0$. Choose any $x_0\in W(p,M^p_0)\cap H(\emptyset, T_1)\cap M^p_0$. By recursion, suppose  we have chosen $\vec x_j = \langle x_0, \ldots,x_{j-1}\rangle\in T_j\cap M^p_0$ so that, for each $i<j$, $x_i\in W(p,M^p_0)$. At step $j$, there is $x_{j}\in H(\vec x_j , T_{j+1})\cap W(p,M^p_0)\cap M^p_0$ because $H(\vec x_j, T_{j+1})\in \mathcal F^+$.  Once we have chosen $\vec x_\ell\in T_\ell\cap M^p_0$, we choose, by elementarity, $r \in D_\ell\cap M$ such that $\vec x_\ell = {\vec x}^r_\ell$.
  \end{proof}

\begin{thm}\label{general-thm-copy-omega1}
 Assume $\mathrm{PFA}$. Let $X$ be a completely regular, countably compact, non-compact space with the property that every time $Y\subset X$ is separable, $C$ is closed in $X$ and $C\subset Y$ then $C$ has a dense set of points with character al most $\omega_1$ in $C$. Then $X$ contains a copy of $\omega_1$.                                                                                                                                                                                                                                           
\end{thm}
\begin{proof}
 
Let $\filt$ be any maximal free closed filter. Define the
function $\varphi$ as follows. For any countable $M\prec H(\kappa)$ such that $X,\mathcal F\in M$, $\varphi(\mbox{Tr}(\mathcal F, M))$ is the set of points of $\mbox{Tr}(\mathcal F, M)$ that have relative character at most
$\omega_1$. For any closed set $B\subset X$ not of this form, let $\varphi(B) = B$.  It should be clear that $\varphi\in H(\kappa)$. By the assumption of the theorem, $\varphi$ is a suitable function. Let $\B$ be the set of all non-empty open subsets of $X$.

Let $\poset=\poset^\varphi_{X,\B,\filt,\kappa}$. We want to
identify $\omega_1$-many dense subsets of $\poset$ so that any filter $G$ meeting them is enough to ensure that $C_G$ is a cub and that $f$ is a homeomorphism. The proof in Proposition \ref{cub} shows that there is a family of $\omega_1$-many dense sets that will guarantee that $C_G$ is a cub. As noted in the proof of Lemma \ref{homeomorphism}, we only need $G$ to capture sufficiently many neighborhoods of each $x_\delta$ in order to ensure that $f$ is continuous.

For each separable $B$ in the domain of $\varphi$ and each $x\in \varphi(B)$, let $\{ U(B,x, \alpha) : \alpha\in \omega_1\}\subset \mathcal U_x$ be chosen so that $\{ B\cap U(B,x,\alpha) : \alpha\in\omega_1\}$ is a local base for $x$ in $B$. For each $\alpha \in \omega_1$, let 
$$
E_\alpha =\{ p\in \poset : (\forall M\in \mathcal M_p\ \exists M\sp\prime\in\mathcal{M}_p\setminus M)
~~ U(\mbox{Tr}(\mathcal F, M),x_{p(M)},\alpha) \in \mathcal U_{p(M\sp\prime)}\}.$$
It is easy to see that $E_\alpha$ is a dense subset of $\poset$.

Now we assume  that $G$ is a filter on $\poset$ and that
$G\cap E_\alpha$ and $G\cap D_{\delta,n}$ are not empty for all  $\alpha,\delta\in \omega_1$ and $n\in \omega$. As usual, let $\{ x_\alpha : \alpha \in  C_G\}$ enumerate the image of $C_G$ by the above mentioned generic function $f$.

As discussed in the proof of Lemma \ref{homeomorphism}, it is enough to prove that $f$ is continuous. Let $\delta\in C_G$ be a limit point of $C_G$ and let $I$ be any cofinal sequence of $C_G\cap \delta$. By the definition of $D_{\delta,n}$, it follows that $\{ \beta \in I : x_\beta \in F(M,n)\}$  is a cofinite subset of $I$. Therefore the set of limit points of $\{ x_\beta : \beta \in I\}$
is a subset of $B = \mbox{Tr}(\mathcal F, M)$. In addition,
for each $\alpha\in \omega_1$, since $G\cap E_\alpha\cap D_{\delta,0}$ is not empty, $\{ x_\beta : \beta\in I\}\setminus U(B,x_\delta,\alpha)$ is finite. It then follows that $x_\delta$ is the unique accumulation point of $\{ x_\beta : \beta\in I\}$. 
\end{proof}

We also include the following which shows that in some cases, we can control the point of convergence of the copy of $\omega_1$.

\begin{coro}\label{coro-free-seq}
 Assume $\mathrm{PFA}$. Let $K$ be a compact Hausdorff space, $X\subset K$ be with the hypothesis of Theorem \ref{general-thm-copy-omega1} and assume that $S\in[X]\sp{\omega_1}$ has all its complete accumulation points in $K\setminus X$. Then there is a copy of $\omega_1$ contained in $X$ that converges to some complete accumulation point of $S$.
\end{coro}
\begin{proof}
 Let $S=\{x_\alpha:\alpha<\omega_1\}$ be an enumeration. For each $\beta<\omega_1$, let $$F_\beta=X\cap\overline{\{x_\alpha:\beta\leq\alpha<\omega_1\}}.$$
 Let $\filt$ be a maximal filter of closed sets that extends $\{F_\alpha:\alpha<\omega_1\}$ and proceed with the proof of Theorem \ref{general-thm-copy-omega1}.
 
 For each $\alpha<\omega_1$, it is easy to prove that the set
 $$
 D\sp\alpha=\{p\in\poset:\exists M\in\mathcal{M}_p\ (F_\alpha\in\filt_{p(M)})\}
 $$
 is dense in $\poset$. So we may assume that $G\cap D\sp\alpha\neq\emptyset$ for all $\alpha<\omega_1$.
 
 If $p\in G\cap D\sp\alpha$ and $M\in\mathcal{M}_p$ is such that $F_\alpha\in\filt_{p(M)}$, then it follows that for any $\gamma\in C_G$ with $M\cap\omega_1<\gamma$, $x_\gamma\in F_\alpha$. Thus,
 $$
 \bigcap_{\alpha<\omega_1}\{x_{p(M)}:\exists p\in G,\exists M\in\mathcal{M}_p\ (\alpha<M\cap\omega_1)\}\subset\bigcap_{\alpha<\omega_1}{F_\alpha},
 $$
 which implies the statement of this Corollary.
\end{proof}

Recall that according to \v Sapirovski's result \cite[3.20, p.71]{juhasz-cardinal_ten_years}, any compact space of size $<2\sp{\omega_1}$ has a dense set of points of character $\leq \omega_1$. A space is $\omega$-bounded if every countable subset has compact closure. Thus, the following result follows from Theorem \ref{general-thm-copy-omega1}.

\begin{coro} 
$\mathrm{PFA}$ implies that any $\omega$-bounded non-compact space of cardinality at most $\mathfrak c$, contains a copy of $\omega_1$. 
\end{coro}

\begin{coro}\label{coro-pseudoradial}
$\mathrm{PFA}$ implies that any pseudoradial space non-sequential space with radial character at most $\aleph_1$ contains a copy of $\omega_1$.
\end{coro}
\begin{proof}
If $K$ is such a space, then $K$ has uncountable tightness. This implies it has a (converging) free  $\omega_1$-sequence.  So $K$ has a subspace $X$ that has a perfect  mapping onto $\omega_1$. With no loss of generality, $X$ has density  equal to $\aleph_1$.  Let $X$ be an element of an elementary  submodel $M$ of $H(\theta)$ (some suitably large $\theta$) such that $M^{\omega_1}\subset M$ and $\lvert M\rvert = 2^{\omega_1} = \omega_2$. Since $K$ is pseudoradial and has radial character $\omega_1$, it follows that $X$ is $\omega$-bounded and is contained in $M$. Now apply the previous Corollary.
\end{proof}

Finally, we prove the result announced in the Introduction.

\begin{proof}[Proof of Theorem \ref{main-positive}]
 Let $p\in X$ and assume that there is a no countable sequence converging to $p$. Since $X$ is almost radial, there is a thin sequence $S=\{x_\alpha:\alpha\in\kappa\}$ converging to $p$. As observed in \cite[Lemma 5.6]{nyikos-handbook}, we may assume that $S$ is free. Since $X$ has radial character $\omega_1$ then $\kappa=\omega_1$. After this, apply the arguments in Corollaries \ref{coro-free-seq} and \ref{coro-pseudoradial} to complete the proof.
\end{proof}

\section{Some final comments and questions}
 
 Let us start this section by making an observation of the proof of Theorem \ref{counterex}. In every step $\omega_1\leq\alpha<\cont$ we can inductively notice that for every $p\in bT$ there is some subset of $\{y_\beta:\beta<\omega_1\}$, below $p$ that has $p$ in its closure. From this, the following follows easily.
  
  \begin{propo}
  For every $\alpha<\omega_1$, let $p_\alpha\in bT$ such that $\pi_\rho(p_\alpha)=y_{\alpha}$. Then the set $\{p_{\alpha+1}:\alpha<\omega_1\}$ is a free $\omega_1$-sequence dense in $bT$. 
  \end{propo}
  
Notice also that we only used $\cont=\omega_2$ to show that the space is pseudoradial but $b=\cont$ is enough for the following.
 
 \begin{thm}\label{thm-bc}
 Assume $\b=\cont$. Then there exists a compact, sequentially compact space $X$ and a continuous function $\pi:X\to\omega_1+1$ such that every time $e:\omega_1+1\to X$ is an embedding there exists $\alpha<\omega_1$ such that $(\pi\circ e)[\omega_1+1]\subset\alpha\cup\{\omega_1\}$.
 \end{thm}
 
Assume $\mathrm{PFA}$. Then the counterexample $X$ from Theorem \ref{counterex} can be constructed and Theorem \ref{main-positive} holds. This means that $X$ cannot satisfy the hypothesis of Theorem \ref{main-positive}: $X$ is either of radial character $\cont=\omega_2$ or $X$ is not almost radial. However, we don't know which one of the two conditions. The only thing we know is that because of Corollary \ref{PFA-positive}, $X$ has character $\omega_2$. So besides Question \ref{main-question} we can also ask the following.

\begin{ques}
 Does it follow from $\mathrm{MA}+\cont=\omega_2$ that there is a almost radial compact Hausdorff space that is not strongly pseudoradial?
\end{ques}

At first, when the authors of this paper attempted the proof of Theorem \ref{counterex}, we intended to kill all copies of $\omega_1$. However, we were not able to give this construction. As we can see from the proof, we were able to kill all copies of $\omega_1$ that converge to the distinguished point $\rho$, but there might exist other copies of $\omega_1$ in other branches of the tree $T$. 
 
From these considerations, one may naturally ask whether it is consistent with $\mathrm{MA}$ that there exists a compact, sequentially compact space of uncountable tightness that contains no topological copies of $\omega_1$. 
 
Consider Nyikos' example of a first countable space $X$ that maps onto $\omega_1$ but has no copies of $\omega_1$. The construction of this example can be found in \cite[19.1]{fremlin}. Essentially, Nyikos example has underlying set $\omega_1\times \{0,1\}$ and it can be easily checked that any subspace of the form $(\alpha+1)\times\{0,1\}$, $\alpha<\omega_1$, is a compact metric space open in $X$. Thus, the one-point compactification of $X$ is sequentially compact, has uncountable tightness and has no copies of $\omega_1+1$. Nyikos' example can be constructed from a statement that is a consequence of $\diamondsuit$ and that is preserved under ccc forcings. We conclude the following.

\begin{coro}
 It is consistent with $\mathrm{MA}$ and $\cont$ of arbitrary large size that there is a compact, sequentially compact space of uncountable tightness that contains no topological copies of $\omega_1+1$.
 \end{coro}
 
 Again, by considering the result by Juh\'asz and Szentmikl\'ossy \cite{juhasz-sz-pseudoradial} that pseudoradiality follows from sequential compactness under $\cont=\omega_2$ we obtain the following.
 
 \begin{coro}
  It is consistent with $\mathrm{MA}+\cont=\omega_2$ that there is a compact pseudoradial space of radial character $\omega_1$ that is not strongly pseudoradial.
 \end{coro}
 
 Besides our Main Question \ref{main-question}, we may also ask the following. 
 
 \begin{ques}
 Does it follow from $\mathrm{MA}+\cont=\omega_2$ that there exists a pseudoradial space (of radial character $\omega_1$) that contains no topological copies of $\omega_1$?
 \end{ques}
 
 \section*{Acknowledgements}

The research in this paper was partially supported by the 2017 PRODEP grant UAM-PTC-636 awarded by the Mexican Secretariat of Public Education (SEP).

\end{document}